\documentclass[12pt]{article}
\def\date{12 Jan 2015. Revised 21 May 2017}
\usepackage{amsmath, amssymb, amsfonts, amsthm, color, epic}

\newtheorem{proposition}{Proposition}[section]
\newtheorem{lemma}[proposition]{Lemma}
\newtheorem{corollary}[proposition]{Corollary}
\newtheorem{theorem}[proposition]{Theorem}
\newtheorem{claim}[proposition]{Claim}
\newtheorem{conjecture}[proposition]{Conjecture}

\newtheorem{thm}[proposition]{Theorem}

{
\theoremstyle{definition}
\newtheorem{definition}[proposition]{Definition}

}

\begin{document}
\font\smallrm=cmr8




\baselineskip=12pt
\phantom{a}\vskip .25in
\centerline{{\bf  On the Minimum Number of Edges in Triangle-Free 5-Critical Graphs}}
\vskip.4in
\centerline{{\bf Luke Postle}
\footnote{\texttt{lpostle@uwaterloo.ca. Partially supported by NSERC under Discovery Grant No. 2014-06162}}} 
\smallskip
\centerline{Department of Combinatorics and Optimization}
\centerline{University of Waterloo}
\centerline{Waterloo, Ontario, Canada N2L 3G1}

\vskip 1in \centerline{\bf ABSTRACT}
\bigskip

{
\parshape=1.0truein 5.5truein
\noindent

Kostochka and Yancey proved that every $5$-critical graph $G$ satisfies: $|E(G)|\ge \frac{9}{4} |V(G)| - \frac{5}{4}$. A construction of Ore gives an infinite family of graphs meeting this bound.

We prove that there exists $\epsilon, \delta > 0$ such that if $G$ is a $5$-critical graph, then $|E(G)| \ge (\frac{9}{4}+\epsilon)|V(G)|-\frac{5}{4} - \delta T(G)$ where $T(G)$ is the maximum number of vertex-disjoint cliques of size three or four where cliques of size four have twice the weight of a clique of size three. As a corollary, a triangle-free $5$-critical graph $G$ satisfies: $|E(G)|\ge  (\frac{9}{4} + \epsilon)|V(G)| - \frac{5}{4}$.
}

\vfill \baselineskip 11pt \noindent \date.
\vfil\eject
\baselineskip 18pt

\section{Introduction}
A graph $G$ is \emph{$k$-critical} if $G$ is not $(k-1)$-colorable but every proper subgraph of $G$ is $(k-1)$-colorable. Since the minimum degree of $k$-critical graph is at least $k-1$, it follows trivially that the number of edges in a $k$-critical graph $G$ is at least $\frac{k-1}{2}|V(G)|$. In a recent landmark paper, Kostochka and Yancey~\cite{KY2} improved this lower bound as follows.

\begin{theorem}[Kostochka, Yancey \cite{KY1}] \label{OreK}
If $G$ is a $k$-critical graph, then $$|E(G)|\ge \left(\frac{k}{2}-\frac{1}{k-1}\right)|V(G)| - \frac{k(k-3)}{2(k-1)}.$$
\end{theorem}

Theorem~\ref{OreK} is tight for all $k$. In particular, $K_k$ satisfies the formula with equality, yet there also exist an infinite family of $k$-critical graphs matching this bound. Theorem~\ref{OreK} confirmed a conjecture of Gallai~\cite{Gallai} on the minimum asymptotic ratio of edges to vertices in a $k$-critical graph and almost proved a conjecture of Ore~\cite{Ore} on the exact lower bound for the number of edges in a $k$-critical graph on a fixed number of vertices. It is natural to wonder whether the lower bound may be improved by restricting to a subclass of $k$-critical graphs. In particular, would excluding certain subgraphs increase the asymptotic ratio of edges to vertices in a $k$-critical graph? 

For $k=4$, Theorem~\ref{OreK} states that a $4$-critical graph $G$ on $n$ vertices has at least $\frac{5n-2}{3}$ edges. A construction of Thomas and Walls~\cite{ThomasWalls} yields an infinite family of $4$-critical triangle-free graphs whose asymptotic ratio of edges to vertices is also $5/3$. In another paper~\cite{Postle}, the author proved the following theorem.

\begin{thm}\label{Girth5}
There exists $\epsilon>0$ such that if $G$ is a $4$-critical graph of girth at least five, then
$$|E(G)|\ge \left(\frac{5}{3}+\epsilon\right)|V(G)|-\frac{2}{3}.$$
\end{thm}

This suggests that excluding certain subgraphs - for $k=4$, the triangle and $4$-cycle - can improve the lower bound. What then is a natural set of subgraphs to exclude for general $k$? The construction of Thomas and Walls can be extended to yield an infinite family of $k$-critical $K_{k-1}$-free graphs whose asymptotic ratio of edges to vertices is $\frac{k}{2} - \frac{1}{k-1}$, the same ratio as in Theorem~\ref{OreK}. 

What if we exclude even smaller cliques? In a work predating Kostochka and Yancey, Krivelevich~\cite{Kri} investigated this question, though the bounds he proved are worse than those of Kostochka and Yancey. Kostochka and Steibitz~\cite{KS} proved that that ratio of edges to vertices in a $k$-critical $K_s$-free graph is at least $k-o(k)$ when $s$ is fixed; this is a multiplicative factor of 2 improvement and is best possible. For details about these and other related results as well as more history about our problem in general, see the extensive survey of Kostochka~\cite{Ksurvey}, in particular Section 5. Also in that section of the survey, Kostochka says that finding the average degree of triangle-free $k$-critical graphs for small and moderate $k$ is an interesting open problem.
 
On the other extreme, instead of excluding fixed sized cliques, it is natural to wonder what is the largest clique that can be excluded such that the tight bound of Kostochka and Yancey can be improved. As noted above, $K_{k-1}$ does not suffice. We make the following conjecture that excluding $K_{k-2}$ does indeed suffice.

\begin{conjecture}\label{K2}
For every $k\ge 4$, there exists $\epsilon_k>0$ such that if $G$ is a $k$-critical $K_{k-2}$-free graph, then 
$$|E(G)|\ge \left(\frac{k}{2}-\frac{1}{k-1} + \epsilon_k\right)|V(G)| - \frac{k(k-3)}{2(k-1)}.$$
\end{conjecture}

The conjecture is vacuously true for $k=4$ since there does not exist a $4$-critical $K_2$-free graph. Hence, Theorem~\ref{Girth5} can be viewed as the appropriate analogue for Conjecture~\ref{K2} with $k=4$ where $K_2$-free is replaced by $\{C_4,K_3\}$-free. The subject of this paper is to consider the case when $k=5$. In fact, we prove Conjecture~\ref{K2} for $k=5$ as follows.

\begin{theorem}\label{Ore5TriangleFree}
There exists $\epsilon > 0$ such that if $G$ is a $5$-critical triangle-free graph, then 
$$|E(G)| \ge \left(\frac{9}{4} + \epsilon\right)|V(G)| - \frac{5}{4}.$$
\end{theorem}

In fact, our main theorem shows that $\epsilon = \frac{1}{84}$ suffices in Theorem~\ref{Ore5TriangleFree}. We should also mention that Conjecture~\ref{K2} has recently been proven for $k=6$~\cite{GaoPostle} and also for large $k$~\cite{Larsen}. Strangely, the range of moderate numbers starting with $k=7$ remains open; indeed, these seem to be the hardest cases. Before we discuss how Theorem~\ref{Ore5TriangleFree} is proved, we need to discuss the family of graphs which attain the bound in Theorem~\ref{OreK} since these are central to the proof of Theorem~\ref{Ore5TriangleFree}.

\begin{definition}
An \emph{Ore-composition} of graphs $G_1$ and $G_2$ is a graph obtained by the following procedure:
\begin{enumerate}
\item delete an edge $xy$ from $G_1$;
\item split some vertex $z$ of $G_2$ into two vertices $z_1$ and $z_2$ of positive degree;
\item identify $x$ with $z_1$ and identify $y$ with $z_2$.
\end{enumerate}
We say that $G_1$ is the \emph{edge-side} and $G_2$ the \emph{vertex-side} of the composition. Furthermore, we say that $xy$ is the \emph{replaced edge} of $G_1$ and that $z$ is the \emph{split vertex} of $G_2$.
We say that $G$ is a \emph{$k$-Ore graph} if it can be obtained from copies of $K_k$ and repeated Ore-compositions.
\end{definition}

Kostochka and Yancey~\cite{KY3} proved that the only graph attaining the bound in Theorem~\ref{OreK} are $k$-Ore graphs as follows.

\begin{theorem}[Kostochka, Yancey \cite{KY3}] \label{OreKEquality}
If $G$ is a $k$-critical graph that is not $k$-Ore, then $$|E(G)|\ge \left(\frac{k}{2}-\frac{1}{k-1}\right)|V(G)| - \frac{y_k}{2(k-1)},$$
where $y_k=\max \{2k-6,k^2-5k+2\}$. 
\end{theorem}

Since $y_5=\max\{4,2\}=4$, this gives the following theorem which we will need in our proof of Theorem~\ref{Ore5TriangleFree}.

\begin{theorem}[Kostochka, Yancey \cite{KY3}] \label{Ore5Equality}
If $G$ is a $5$-critical graph that is not $5$-Ore, then $$|E(G)|\ge \frac{9}{4}|V(G)| - \frac{1}{2}.$$
\end{theorem}

To prove Theorem~\ref{Ore5TriangleFree}, we would also need to prove its corollary that every $5$-Ore graph contains a triangle. In fact, we prove much more. We show that $5$-Ore graphs contain linearly many vertex-disjoint triangles. The concept of tracking not just whether a graph contains a triangle but how many vertex-disjoint triangles it has is actually the crucial idea for the proof. This was also the key idea in the proof of Theorem~\ref{Girth5} where the number of vertex-disjoint cycles of length at most four was tracked. Here, we will also need to track copies of $K_4$, which while not containing two vertex-disjoint triangles are more valuable structurally than just a triangle. To that end, we make the following definition.

\begin{definition}
If $H$ is a disjoint union of cliques of size three or four, then we let $T(H)$ be the number of components in $H$ that are cliques of size three plus twice the number of components which are cliques of size four. More generally, we let $T(G)$ denote the maximum of $T(H)$ over all subgraphs $H$ of $G$ that are the disjoint union of cliques of size three or four.
\end{definition}

We are now ready to state our main result which proves that the lower bound on the asymptotic ratio of edges to vertices in $k$-critical graph $G$ may be increased if a factor proportional to $T(G)$ is subtracted.

\begin{thm}\label{Main}
Let $\epsilon=\frac{1}{21}$, $\delta=8\epsilon$ and $P=6\delta = 48\epsilon$. Let $$p(G) = (9+\epsilon)|V(G)|-4|E(G)|-\delta T(G).$$ If $G$ is a $5$-critical graph, then 

\begin{enumerate}
\item $p(G) = 5 + 5\epsilon-2\delta$ if $G=K_5$,
\item $p(G)\le 5 + |V(G)|\epsilon - (2+\frac{(|V(G)|-1)}{4})\delta$ if $G$ is $5$-Ore and $G\ne K_5$,
\item $p(G)\le 5 - P$ otherwise. 
\end{enumerate}
\end{thm}

Note that Theorem~\ref{Main} proves three different bounds, one for $K_5$, one for the other $5$-Ore graphs and one for all remaining graphs. This is necessary for the inductive step of the proof to work. As for the first bound, $T(K_5)=2$ and hence $p(K_5)=5+5\epsilon -2\delta$ as desired. 

In Section 2, we prove the following lemma.

\begin{lemma}\label{5OrePot}
If $G\ne K_5$ is a $5$-Ore graph, then $T(G)\ge 2 + \frac{|V(G)|-1}{4}$.
\end{lemma}

Hence if $G$ is $5$-Ore then $p(G) \le 5 + |V(G)|\epsilon - (2+\frac{(|V(G)|-1)}{4})\delta$ since $9|V(G)|-4|E(G)|=5$ when $G$ is $5$-Ore. This proves the second assertion in Theorem~\ref{Main}. In Sections 3-6, we complete the proof of Theorem~\ref{Main} by proving that if $G$ is a $5$-critical graph that is not $5$-Ore, then $p(G)\le 5-P$ for a $P$ to be determined later.

We note that Theorem~\ref{Main} is a much stronger theorem than Theorem~\ref{Ore5TriangleFree} in that it shows that a graph whose ratio of edges to vertices is below that of Theorem~\ref{Ore5TriangleFree} has linearly many vertex-disjoint triangles. Why do we prove this stronger theorem? Because we use the potential method of Kostochka and Yancey whose key reduction of identifying vertices in a colored subgraph may create triangles. Hence to use the potential method we must prove a theorem which holds for all $5$-critical graphs not just $5$-critical triangle-free graphs. Indeed, it is this which motivates Theorem~\ref{Main} and the definition of $T(G)$. Furthermore, this explains the choice of vertex-disjoint in the definition of $T(G)$ as opposed to edge-disjoint. The reduction of Kostochka and Yancey may create triangles but it creates at most four new vertex-disjoint triangles while it may create arbitrarily many new edge-disjoint triangles.

Here is an outline of the paper. In Section 2, we prove that $T(G)$ satisfies a certain inequality (Lemma~\ref{5OreT}) for Ore-compositions and use that to prove Lemma~\ref{5OrePot}. We further prove a number of structural properties of $5$-Ore graphs which we will need for the general proof. In Section 3, we extend the notion of Kostochka and Yancey's potential to a new potential which incorporates $T(G)$. We show that this new potential satisfies Kostochka and Yancey's submodular inequality (Lemma~\ref{Extension}) for their key reduction up to a small additive error. We also characterize under what circumstances said inequality is tight. In Section 4, we prove that a minimal counterexample to Theorem~\ref{Main} as well as any graph whose potential is close to being a counterexample (what we call \emph{tight} graphs) satisfies certain structural properties. In Section 5, we use these properties of tight graphs to show that a minimal counterexample satisfies an even stronger list of properties. Finally in Section 6, we prove Theorem~\ref{Main} using discharging.

\section{Cliques in $5$-Ore graphs}

In this section, we investigate cliques of size three and four in $5$-Ore graphs. We also prove Lemma~\ref{5OrePot}. First note the following observation.

\begin{lemma}\label{5OreT}
If $G$ is the Ore-composition of two graphs $G_1$ and $G_2$, then $T(G)\ge T(G_1)+T(G_2)-2$. Furthermore if $G_2=K_5$, then $T(G)\ge T(G_1)+1$.
\end{lemma}
\begin{proof}
To prove the first statement, without loss of generality let $e$ be the replaced edge of $G_1$ and $z$ the split vertex of $G_2$. It follows that $T(G)\ge T(G_1-e) + T(G_2\setminus z)$. But $T(G_1)-e\ge T(G_1)-1$ and $T(G_2\setminus z)\ge T(G_2) - 1$. Hence $T(G)\ge T(G_1)+T(G_2)-2$ as desired.

To prove the second statement, note that for every edge $e\in E(K_5)$, $T(K_5-e)=2$ and for every vertex $z\in V(K_5)$, $T(K_5\setminus z)=2$. Thus in either case, it follows from the calculations above that $T(G)\ge T(G_1)-1 + 2 = T(G_1)+1$.
\end{proof}

We are now ready to prove Lemma~\ref{5OrePot}.

\vskip .1in
{\bf Proof of Lemma~\ref{5OrePot}.} We proceed by induction on $|V(G)|$. Since $G\ne K_5$ and $G$ is $5$-Ore, $G$ is the Ore-composition of two graphs $G_1$ and $G_2$. For each $i\in\{1,2\}$, if $G_i\ne K_5$, then by induction $T(G_i)\ge 2+\frac{|V(G_i)|-1}{4}$. 

First suppose that neither $G_1$ nor $G_2$ is isomorphic to $K_5$. By Lemma~\ref{5OreT}, $T(G)\ge T(G_1) + T(G_2) - 2$. By induction applied to $G_1$ and $G_2$, we find that $T(G)\ge 2+\frac{|V(G_1)|-1}{4} + 2 + \frac{|V(G_2)|-1}{4} - 2 = 2 + \frac{|V(G_1)|+|V(G_2)|-2}{4}$. Yet $|V(G)|=|V(G_1)|+|V(G_2)|-1$. So $T(G)\ge 2+\frac{|V(G)|-1}{4}$ as desired.

So we may assume without loss of generality that $G_2=K_5$. Next suppose $G_1\ne K_5$. By Lemma~\ref{5OreT}, $T(G)\ge T(G_1)+1$. By induction, $T(G_1)\ge 2+\frac{|V(G_1)|-1}{4}$. So $T(G)\ge 3+\frac{|V(G_1)|-1}{4}$. Yet $|V(G)|=|V(G_1)|+4$, so $T(G)\ge 2+\frac{|V(G)|-1}{4}$ as desired.

Finally suppose both $G_1$ and $G_2$ are isomorphic to $K_5$. Without loss of generality, suppose that $G_1$ is the edge-side and $G_2$ the vertex side of the composition and let $e$ be the replaced edge of $G_1$ and $z$ the split vertex of $G_2$. Then $T(G_1-e)=T(G_2\setminus z)=2$ as both $G_1$ and $G_2$ contain $K_4$ as a subgraph. Hence $T(G)\ge T(G_1-e)+T(G_2\setminus z) = 2 + 2 = 4$. Meanwhile, $|V(G)|=5+5-1=9$. Thus, $T(G)=4\ge 2+\frac{9-1}{4}$ as desired.
\qed

\vskip .1in
We also need the following lemmas about the structure of $k$-Ore graphs. First a few definitions.

\begin{definition}
A \emph{diamond} in a graph $G$ is a subgraph isomorphic to $K_5-e$ where the vertices not incident to $e$ have degree four in $G$. An \emph{emerald} in a graph $G$ is a subgraph isomorphic to $K_4$ whose vertices have degree four in $G$. A graph is \emph{ungemmed} if it contains neither a diamond nor an emerald.
\end{definition}

\begin{definition}
An \emph{Ore-collapsible} subset $R$ of $V(G)$ is a proper subset such that the boundary of $R$ contains exactly two non-adjacent vertices $u,v$ and $R+uv$ is $5$-Ore.
\end{definition}

\begin{lemma}\label{OreCollOrEmeraldVertex}
If $H$ is $5$-Ore and $v\in V(H)$, then there exists either an Ore-collapsible subset of $H$ not containing $v$ or an emerald of $H$ not containing $v$.
\end{lemma}

\begin{proof}
We proceed by induction on $|V(H)|$. If $H=K_5$, then every vertex is disjoint from an emerald as desired. So we may suppose that $H\ne K_5$. As $H$ is $5$-Ore, then $H$ is the Ore-composition of two $5$-Ore graphs $H_1$ and $H_2$. Without loss of generality suppose that $H_1$ is the edge-side and $H_2$ is the vertex-side of the composition. We now have two cases: either $v\in V(H_1)$ or $v\in V(H_2)\setminus V(H_1)$.

First suppose $v\in V(H_1)$. Let $z$ be the split vertex of $H_2$. By induction, there exists an emerald or Ore-collapsible subset of $H_2$ not containing $z$. But then there exists an Ore-collapsible subset or emerald of $H$ not containing $v$ as desired. So we may suppose that $v\in V(H_2)\setminus V(H_1)$. But then $V(H_1)$ is an Ore-collapsible subset of $H$ not containing $v$ as desired.
\end{proof}

\begin{lemma}\label{OreCollOrEmeraldK4}
If $H\ne K_5$ is $5$-Ore and $T=K_4$ is a subgraph of $H$, then there exists either an Ore-collapsible subset of $H$ disjoint from $T$ or an emerald of $H$ disjoint from $T$.
\end{lemma}

\begin{proof}
We proceed by induction on $|V(H)|$. As $H\ne K_5$ is $5$-Ore, then $H$ is the Ore-composition of two $5$-Ore graphs $H_1$ and $H_2$. Without loss of generality suppose that $H_1$ is the edge-side with replaced edge $xy$ and $H_2$ is the vertex-side of the composition with split vertex $z$. We now have two cases since $x$ and $y$ are non-adjacent in $H$: either $T\subseteq H_1$ or $T\subseteq H_2$.

First suppose $T\subseteq H_1$. By Lemma~\ref{OreCollOrEmeraldVertex}, there exists either an Ore-collapsible subset or an emerald of $H_2$ not containing $z$. But then there exists an Ore-collapsible subset or emerald of $H$ disjoint from $T$ as desired. 

So we may suppose that $T\subseteq H_2$. Let $xy$ be the replaced edge of $H_1$. As $x$ and $y$ are not adjacent in $H$, we may suppose without loss of generality that $y\not\in T$. If $x\not\in T$, then $V(H_1)$ is an Ore-collapsible subset disjoint from $T$ as desired. So we may suppose that $x\in T$. Notice that $T'=T-x\cup z$ is a subgraph of $H_2$ isomorphic to $K_4$. If $H_2\ne K_5$, then by induction, there exists an Ore-collapsible subset of $H_2$ disjoint from $T'$ or an emerald of $H_2$ disjoint from $T'$. But then there exists an Ore-collapsible subset or emerald of $H$ disjoint from $T$ as desired. So we may suppose that $H_2=K_5$. Thus $y$ has only one neighbor in $V(H_2)\setminus V(H_1)$ and so has the same degree in $H$ as it does in $H_1$. By Lemma~\ref{OreCollOrEmeraldVertex}, there exists either an Ore-collapsible subset or an emerald of $H_1$ disjoint from $x$. If there exists an Ore-collapsible subset $R$ of $H_1$ disjoint from $x$, then $R$ is also an Ore-collapsible subset of $H$ disjoint from $T$ as desired. So we may assume there exists an emerald $E$ of $H_1$ disjoint from $x$. But all vertices of $H_1\setminus\{x\}$ have the same degree in $H_1$ as in $H$. Thus $E$ is also an emerald of $H$ and is disjoint from $T$ as desired.
\end{proof}

\begin{lemma}\label{OreCollToGem}
If $R$ is an Ore-collapsible subset of a graph $G$, then there exists a diamond or emerald of $G$ whose vertices lie in $R$.
\end{lemma}
\begin{proof}
We proceed by induction on $|R|$. Let $u,v$ be the boundary vertices of $R$ and let $H=R+uv$. If $H=K_5$, then $R$ is a diamond of $G$ as desired. So we may assume that $H\ne K_5$ and hence $H$ is the Ore-composition of two $5$-Ore graphs $H_1$ and $H_2$. Without loss of generality suppose that $H_1$ is the edge-side of the composition with replaced edge $xy$ and $H_2$ is the vertex-side of the composition with split vertex $z$. 

If $u,v\in V(H_1)$, then by Lemma~\ref{OreCollOrEmeraldVertex}, there exists either an Ore-collapsible subset or an emerald of $H_2$ disjoint from $z$. If there exists an Ore-collapsible subset $R'$ of $H_2$ disjoint from $z$, then $R'$ is also an Ore-collapsible subset of $G$ with $|R'| < |R|$. By induction, there exists a diamond or emerald of $G$ whose vertices lie in $R'$ and hence in $R$ as desired. If there exists an emerald $E$ of $H_2$ disjoint from $z$, then $E$ is also an emerald of $G$ whose vertices lie in $R$ as desired.

So we may assume that $u,v\in V(H_2)$. But then $V(H_1)$ is an Ore-collapsible subset of $H$ and hence of $G$ where $|V(H_1)|<|R|$. By induction, there exists a diamond or emerald of $G$ whose vertices lie in $V(H_1)$ and hence in $R$ as desired.
\end{proof}

Combining the lemmas above gives the following result.

\begin{lemma}\label{DisjointGem}
Let $H$ be $5$-Ore. For any vertex $v\in V(H)$, there exists a diamond or emerald of $H$ disjoint from $v$. If $H\ne K_5$, then for any subgraph $T$ of $H$, there exists a diamond or emerald disjoint from $T$.
\end{lemma}
\begin{proof}
Follows from Lemmas~\ref{OreCollOrEmeraldVertex}, \ref{OreCollOrEmeraldK4}, and \ref{OreCollToGem}.
\end{proof}

\section{Potential}

We follow Kostochka and Yancey's proof~\cite{KY2} of Theorem~\ref{OreK} for $k=5$. A key concept for the proof is that of a potential function for subgraphs. For $k=5$, Kostochka and Yancey's potential is as follows. We also define our version of potential as in Theorem~\ref{Main} which incorporates $T(G)$.

\begin{definition}
The \emph{Kostochka-Yancey potential} of a graph $G$ , denoted $p_{KY}(G)$, is $9|V(G)|-4|E(G)|$. The \emph{potential} of a graph $G$, denoted $p(G)$, is $(9+\epsilon)|V(G)|-4|E(G)|-\delta T(G)$. If $R\subseteq V(G)$, then we define $p_{KY}(R) = p_{KY}(G[R])$ and $p_G(R)=p(G[R])$. 
\end{definition}

Theorem~\ref{OreK} for $k=5$ can be restated as follows.

\begin{theorem}\label{Ore5P}
If $G$ is $5$-critical, then $p_{KY}(G) \le 5$.
\end{theorem}

Similarly Theorem~\ref{Ore5EqualityP} may be restated as follows.

\begin{theorem}\label{Ore5EqualityP}
If $G$ is $5$-critical, then $p_{KY}(G)\ge 3$ if and only if $G$ is $5$-Ore.
\end{theorem}

Here is the key reduction to be used with potential.

\begin{definition}
If $R\subsetneq V(G)$ with $|R|\ge 5$, and $\phi$ is a $4$-coloring of $G[R]$, we define the \emph{$\phi$-identification of $R$ in $G$}, denoted by $G_{\phi}(R)$, to be the graph obtained from $G$ by identifying for each $i\in\{1,2,3,4\}$ the vertices colored $i$ in $R$ to a vertex $x_i$, adding the edges $x_ix_j$ for all $i,j\in\{1,2,3,4\}$ and then deleting parallel edges. 
\end{definition}

\begin{proposition}[Claim 8~\cite{KY2}]\label{Phi}
If $G$ is $5$-critical, $R\subsetneq V(G)$ with $|R|\ge 5$, and $\phi$ is a $4$-coloring of $G[R]$, then $\chi(G_{\phi}(R))\ge 5$.
\end{proposition}

Hence $G_{\phi}(R)$ contains a $5$-critical graph and we may extend the set $R$ to a larger set as follows:

\begin{definition}
Let $G$ be a $5$-critical graph, $R\subsetneq V(G)$ with $|R|\ge 5$ and $\phi$ a $4$-coloring of $G[R]$. Now let $W$ be a $5$-critical subgraph of $G_{\phi}(R)$ and $X$ be the graph on the set of vertices $x_i$. Then we say that $R' = (V(W)-V(X))\cup R$ is the \emph{critical extension} of $R$ with \emph{extender} $W$. We call $W\cap X$ the \emph{core} of the extension. 

If in $G$ a vertex in $W-V(X)$ has more neighbors in $R$ than in $V(W \cap X)$, or there exists an edge in $G[V(W)-V(X)]$ that is not in $W-V(X)$, or $W[X]$ is not a complete graph, then we say that the extension is \emph{incomplete}. Otherwise, we say the extension is \emph{complete}. If $R'=V(G)$, we say the extension is \emph{spanning}. If the extension is both complete and spanning, then we say it is \emph{total}.
\end{definition}

Note that every critical extension has a non-empty core as otherwise $G$ would contain a proper non-$4$-colorable subgraph contradicting that $G$ is $5$-critical. 

Kostochka and Yancey proved the following key lemma about their potential in regards to critical extensions.

\begin{lemma}\label{KYExtension}
If $G$ is a $5$-critical graph, $R\subsetneq V(G)$ with $|R|\ge 5$ and $R'$ is a critical extension of $R$ with extender $W$ and core $X$, then
$$p_{KY}(R')\le p_{KY}(R) + p_{KY}(W) - 9/14/15/12$$ when $|X|$ is $1/2/3/4$ respectively.
\end{lemma}

Here is the corresponding lemma bounding our potential for critical extensions in terms of the original set and the extender. Note the use of the vertex-disjointness of $T(G)$.

\begin{lemma}\label{Extension}
If $G$ is a $5$-critical graph, $R\subsetneq V(G)$ with $|R|\ge 5$ and $R'$ is a critical extension of $R$ with extender $W$ and core $X$, then 

$$p_G(R')\le p_G(R) + p(W) - f(|X|) + \delta(T(W)-T(W\setminus X)),$$
where $f(|X|)=p(K_{|X|})-T(X)$.\\

Furthermore, 
$$p_G(R')\le p_G(R)+p(W) -9-\epsilon+\delta.$$
\end{lemma}

\begin{proof}
Each vertex of $G[R']$ is a vertex of $G[R]$ or $W\setminus X$, while each of edge of $G[R]$ and $W-E(G_{\phi}[X])$ is an edge of $G[R']$. 
So $\lvert R' \rvert = \lvert R \rvert + \lvert V(W) \rvert - \lvert X \rvert$, and $\lvert E(G[R']) \rvert \geq \lvert E(G[R]) \rvert + \lvert E(W) \rvert - {{\lvert X \rvert}\choose{2}}$.
Furthermore, $T(R')\ge T(R)+T(W\setminus X)$.
Therefore, 
$$p_G(R') \leq p_G(R)+p(W) - (9+\epsilon)\lvert X \rvert + 4 {\lvert X \rvert \choose 2} + \delta(T(W) - T(W \setminus X)).$$
Note that $f(X) = 9\lvert X \rvert - 4 {\lvert X \rvert \choose 2} +|X|\epsilon$.
Observe that $T(W)-T(W\setminus X) \le |X|$. 
Hence $p_G(R')\le p_G(R)+p(W)-9-\epsilon+\delta$ since $\epsilon \le \delta \le 3$.
\end{proof}

Note that $f(1)=9+\epsilon$, $f(2) = 14+2\epsilon$, $f(3)=15+3\epsilon$, and $f(4)=12+4\epsilon$. 

\subsection{Collapsible Sets}

Here is a crucial definition.

\begin{definition}
Let $G$ be a graph and $R\subsetneq V(G)$ such that $|R|\ge 5$. The \emph{boundary} of $R$ is the set of vertices in $R$ with neighbors in $G-R$. If $G$ is $5$-critical, we say $R$ is \emph{collapsible} if in every $4$-coloring of $G[R]$ all vertices in the boundary of $R$ receive the same color. If $R$ is collapsible, then we define the \emph{critical complement} of $R$ to be the graph obtained by identifying the boundary of $R$ to one vertex $v$ and deleting the rest of $R$. We call $v$ the \emph{special vertex} of $W$.

Note then that the boundary of $R$ is an independent set and for any $u,v$ in the boundary of $R$, $G[R]+uv$ contains a $5$-critical subgraph. We say a collapsible subset is \emph{tight} if for any $u,v$ in the boundary of $R$, $G[R]+uv$ is $5$-critical.
\end{definition}

Here is an easy proposition.

\begin{proposition}\label{CritComp}
If $R$ is collapsible, then the critical complement $W$ of $R$ is $5$-critical.
\end{proposition}

Indeed $V(G)$ is a $W$-critical extension of $R$ and that extension has a core of size one and is complete. Furthermore, this is the only critical extension of $R$. The converse is also true as the next proposition shows.

\begin{proposition}\label{Coll}
$R$ is collapsible if and only if for every critical extension $R'$, the extension is complete, spanning and has a core of size one.
\end{proposition}
 
As we shall see in the next section, we are interested in proper subgraphs of relatively smallest potential. What properties do such subsets satisfy? Well, given Lemma~\ref{Extension}, such subgraphs must have extensions yielding a minimum decrease, and hence all of the subgraph's extensions must have cores of size one. Moreover if such a set had an extension that is incomplete, then $p_G(R')\le p_G(R)+p(W) - 13-\epsilon+\delta$, leading to larger decrease. So we may assume all extensions of are complete as well. Furthermore, assuming Theorem~\ref{Main} is true inductively, extensions decrease potential by about 4. Hence all extension of such a set are spanning, as otherwise the extension itself would have smaller potential. Thus all extensions of such sets are complete, spanning and have cores of size one. By Propisition~\ref{Coll}, this is equivalent to being collapsible; hence the importance of collapsible sets.

As for the potential of collapsible sets, we now have the following statement.

\begin{proposition}\label{CollPot}
If $R$ is a collapsible subset of a $5$-critical graph $G$ and $W$ is the critical complement of $R$, then 
$$p_G(R)\ge p(G)-p(W)+9+\epsilon-\delta$$
\end{proposition}
\begin{proof}
Follows from Lemma~\ref{Extension} since $V(G)$ is the $W$-critical extension of $R$ which is complete and has a core of size one by Proposition~\ref{Coll}.
\end{proof}

Here is another nice corollary of Proposition~\ref{Coll} about subsets of $5$-Ore graphs which we will need later.

\begin{lemma}\label{Coll5Ore}
If $G$ is $5$-Ore and $R\subsetneq V(G)$ such that $|R|\ge 5$ and $p_{KY}(R) < 12$, then $R$ is collapsible and hence $p_{KY}(R)=9$.
\end{lemma}
\begin{proof}
Suppose for a contradiction that $R$ is not collapsible. Note that it follows from Lemma~\ref{KYExtension} that $p_{KY}(S)\ge 9$ for every $S\subsetneq V(G)$. Now by Proposition~\ref{Coll} since $|R|\ge 5$, there exists a critical extension $R'$ of $R$ such that the extension is either incomplete, has a core of size at least two or $R'\ne V(G)$. First suppose $R'\ne V(G)$. Then $p_{KY}(R')\ge 9$ as noted above. By Lemma~\ref{KYExtension}, $p_{KY}(R)\ge p_{KY}(R')+9-p_{KY}(W)$. By Theorem~\ref{Ore5P}, $p_{KY}(W)\le 5$ and hence $p_{KY}(R)\ge 9+9-5 = 13$, a contradiction. So we may suppose $R'=V(G)$. Next suppose that the extension is incomplete. Then $p_{KY}(R') \le p_{KY}(R) + p_{KY}(W) - 9 - 4$.  Hence $p_{KY}(R)\ge p_{KY}(R')+13 -p_{KY}(W)$. Now $p_{KY}(R')=p_{KY}(G)=5$ while $p_{KY}(W)\le 5$ by Theorem~\ref{Ore5P}. Thus $p_{KY}(R)\ge 13$, a contradiction. So we may assume that the extension has a core of size at least two. By Lemma~\ref{KYExtension}, $p_{KY}(R') \le p_{KY}(R) + P_{KY}(W) - 12$. Hence, $p_{KY}(R)\ge p_{KY}(R') - p_{KY}(W) + 12$. Yet $p_{KY}(R')=p_{KY}(G)=5$ while $p_{KY}(W)\le 5$ by Theorem~\ref{Ore5P}. Thus $p_{KY}(R)\ge 12$, a contradiction.  
\end{proof}

\section{Structures of a Minimum Counterexample}
In this section, we attempt to understand the structure of a minimum counterexample to Theorem~\ref{Main}. 

For the proof, we will also need to understand the structure of graphs which are not counterexamples but are close in potential to a counterexample. To that end, we make the following definitions.

\begin{definition}
A \emph{cluster} in a graph $G$ is a maximal set of degree four vertices with the same closed neighborhood.

A graph $H$ is \emph{smaller} than a graph $G$ if either $|V(H)|< |V(G)|$, or $|V(G)|=|V(H)|$ and $|E(H)|>|E(G)|$, or, $|V(G)|=|V(H)|$ and $|E(G)|=|E(H)|$ and $H$ precedes $G$ in the lexicographical ordering of clusters sizes by decreasing value.
\end{definition}

\begin{definition}
A $5$-critical graph $G$ is \emph{good} if every $5$-critical graph smaller than $G$ satisfies Theorem~\ref{Main}. 
Let $Q=\delta$. A graph $G$ is \emph{tight} if $G$ is good and $p(G)\ge 5-P-Q$.
\end{definition}

\begin{lemma}\label{ExtDec}
If $G$ is good, $R\subsetneq V(G)$, $|R|\ge 5$ and $R'$ is a critical extension of $R$, then $p_G(R)\ge p_G(R')+4-\delta+4\epsilon$. Furthermore $p_G(R')\ge p(G)$ and hence $p_G(R)\ge p(G)+4-\delta+4\epsilon$.
\end{lemma}

\begin{proof}
Suppose that $R'$ is a critical extension with extender $W$. 
As $G$ is good, $p(W)\le 5 +5\epsilon-2\delta$. 
By Lemma~\ref{Extension}, $p_G(R')\le p_G(R) + p(W) -  9 + \delta - \epsilon \le p_G(R)  - 4 + 4\epsilon-\delta$ as desired. 
By repeatedly applying this result to further critical extensions, we find that $p(G)\le p_G(R') \le p_G(R)-4+4\epsilon-\delta$.
\end{proof}

\begin{lemma}\label{PotCore1}
If $G$ is tight, $R\subsetneq V(G)$, $|R|\ge 5$ and $p_G(R)< p(G)+7-4\delta+4\epsilon$, then $R$ is collapsible.
\end{lemma}

\begin{proof}
As $R$ is a proper subset of $V(G)$ with $|R|\ge 5$, $R$ has a critical extension $R'$ with extender $W$ and core $X$. 
By Lemma~\ref{Extension}, $p_G(R')\le p_G(R)+p(W)-f(|X|) + \delta(T(W)- T(W\setminus X))$. By Lemma~\ref{ExtDec}, $p(G)\le p_G(R')$.

If the extension is not spanning, then $p_G(R')\ge p(G)+4-\delta+4\epsilon$ and hence $p_G(R)\ge p(G) + 8-2\delta + 8\epsilon$, a contradiction. So we may suppose the extension is spanning. If the extension is not complete, then $p(G)=p_G(R') \le p_G(R)+p(W) - 9 + \delta-\epsilon - 4 \le p_G(R) - 8 -\delta + 4\epsilon$, a contradiction. So we may suppose the extension is total. If the extension has a core of size at least two, then $p(G)=p_G(R')\le p_G(R) + p(W) - 12 + 4\delta - 4\epsilon$ (and only then when $|X|=4$). Hence, $p(G)\le p_G(R) -7 + 4\delta-4\epsilon$, a contradiction.

As the extension $R'$ was arbitrary, this implies that every critical extension of $R$ is spanning, complete and have a core of size one. By Lemma~\ref{Coll}, $R$ is collapsible.
\end{proof}

\begin{lemma}\label{CollSmall}
Let $G$ be tight, $R\subsetneq V(G)$, and $u,v\in R$. If $K=R+uv$ is $5$-critical, then $R$ is collapsible and either
\begin{enumerate}
\item $K$ is $5$-Ore, or,
\item the critical complement $W$ of $R$ is $5$-Ore.
\end{enumerate}
\end{lemma}
\begin{proof}
Note that $p(R) \le p(K)+4+\delta$. As $G$ is good, $p(K)\le 5-2\delta+5\epsilon$. Hence $p(R)\le 9-\delta+5\epsilon$ and it follows from Lemma~\ref{PotCore1} that $R$ is collapsible. Suppose that neither $K$ nor the critical complement $W$ of $R$ is $5$-Ore. Then $p(K),p(W)\le 5-P$ as $G$ is good. Thus by Lemma~\ref{Extension}, $p(G)\le p_G(R) + p(W) - 9 + \delta-\epsilon \le p(K) + p(W) - 5 +2\delta-\epsilon \le 5 - 2P + 2\delta -\epsilon$. This contradicts that $G$ is tight as $P-2\delta + \epsilon = 4\delta-2\epsilon > \delta = Q$. Thus either $K$ or $W$ is $5$-Ore as desired.
\end{proof}


\begin{definition}
We say $u,v\in V(G)$ is an \emph{identifiable pair} in a proper subset $R$ of $V(G)$ if $G[R]+uv$ is not $4$-colorable.
\end{definition}

We will we show that tight ungemmed graphs do not contain an identifiable pair but first we need the following lemma.

\begin{lemma}\label{5OreColl}
If $R$ is a proper collapsible subset in a graph $G$ and $u,v$ are in the boundary of $R$ and $R+uv$ is $5$-Ore, then there exists a subset $R'$ of $R$ such that:
\begin{enumerate}
\item either $R'$ is an Ore-collapsible subset of $G$, or, 
\item $R'$ is a proper subset of $R$ and there exists $u',v'$ in the boundary of $R'$ such that $R'+u'v'$ is $5$-Ore.
\end{enumerate}
\end{lemma}
\begin{proof}
If the boundary of $R$ consists only of $u$ and $v$, then $R$ is an Ore-collapsible subset of $G$ and 1 holds as desired. So we may suppose that there exists a third vertex $w\ne u,v$ in the boundary of $R$. As $R$ is collapsible, $w$ must receive the same color as $u$ and $v$ in every $4$-coloring of $R$. This implies that $H\ne K_5$ where $H=R+uv$. Since $H$ is $5$-Ore and $H\ne K_5$, $H$ is the Ore-composition of two $5$-Ore graphs $H_1$ and $H_2$. Without loss of generality suppose that $H_1$ is the edge-side with replaced edge $xy$ and $H_2$ is the vertex-side of the composition with split vertex $z$.

If $u,v\in V(K_2)$, then $R'=V(K_1)$, $u'=x$, $v'=y$ yields a $5$-Ore graph $R'+u'v'$ and 2 holds as desired. So we may assume that $u,v\in V(K_1)$. But this holds for every Ore-composition yielding $H$. Thus $H$ is obtained from a graph $J\cong K_5$ containing $u$ and $v$ by Ore-compositions whose replaced edges are in $E(J)$. 

Note $w\not\in V(J)$ as $w$ must receive the same color as $u$ and $v$ which the other vertices of $J$ do not. Thus $w$ lies on the vertex-side, call it $S$, of one of these replaced edges of $E(J)$, call it $e$. However, $e\ne uv$ and hence $e$ is incident with at most one of $u$ or $v$. Suppose without loss of generality that $e$ is not incident with $v$. 

\begin{claim}\label{uw}
$u,w$ is an identifiable pair in $R-v$.
\end{claim}
\begin{proof} 
Suppose not and let $\phi$ be $4$-coloring of $R-v + uw$. Note that the four vertices of $V(J)-v$ must all receive different colors in $\phi$ since $R-v$ contains a subgraph that can be obtained from $J-v$ by Ore-compositions whose replaced edges are in $E(J-v)$. Now let $\phi'$ be a $4$-coloring of $R$. Once again the four vertices of $V(J)-v$ must all receive different colors in $\phi$. We may assume without loss of generality by permuting colors as necessary, that $\phi(x)=\phi'(x)$ for all $x\in V(J)-v$. Now let $\phi''(y) = \phi(y)$ if $y \in S$ and $\phi''(y) = \phi'(y)$ if $y\not\in S$. Note that $\phi''$ is a proper $4$-coloring of $R$ and yet $\phi''(u)=\phi'(u)=\phi(u)\ne \phi(w) = \phi''(w)$, a contradiction since $u$ and $w$ are in the boundary of $R$ and $R$ is collapsible. 
\end{proof}

By Claim~\ref{uw}, $u,w$ is an identifiable pair in $R-v$; that is, $R-v+uw$ is not $4$-colorable and hence contains a $5$-critical subgraph $K$. By Lemma~\ref{Coll5Ore}, it follows that $p_{KY}(K-uw)\ge 9$ since $K-uw$ is a proper subgraph of the $5$-Ore graph $H$. But then $p_{KY}(K) = p_{KY}(K-uw) - 4\ge 5$. By Lemma~\ref{Ore5EqualityP}, $K$ is $5$-Ore. Let $R'=V(K)$ and note that since $u$ and $w$ are in the boundary of $R$ in $G$, it follows that $u$ and $w$ are in the boundary of $R'$ in $G$. Let $u'=u$ and $v'=w$. Thus $R'+u'v'=K$ is $5$-Ore, $R'\subseteq R-v$, and $u',v'$ are in the boundary of $R'$ and 2 holds as desired.
\end{proof}
 
\begin{lemma}\label{NoIdPair}
If $G$ is a tight ungemmed graph, then there does not exist an identifiable pair in a proper subset of $V(G)$.
\end{lemma}
\begin{proof}
Suppose not. Let $u,v$ be an indentifiable pair in some proper subset $R$ of $V(G)$ chosen such that $|R|$ is minimized and subject to that, $u,v$ are in the boundary of $R$ if possible. We may assume without loss of generality that $K=R+uv$ is a $5$-critical graph. By Lemma~\ref{CollSmall}, $R$ is collapsible and either $K$ is $5$-Ore or the critical complement $W$ of $R$ is $5$-Ore. As $R$ is collapsible, every pair of vertices on its boundary is an identifiable pair. So we may assume that $u,v$ are in the boundary of $R$ by the choice of $R$. If $W$ is $5$-Ore, then by Lemma~\ref{DisjointGem}, it follows that there exists a diamond or an emerald disjoint from the special vertex of $W$, contradicting that $G$ is ungemmed. So we may assume that $K$ is $5$-Ore. 

By the minimality of $R,u$ and $v$, the first outcome of Lemma~\ref{5OreColl} does not hold and hence the second outcome holds. That is, there exists a subset $R'$ of $R$ that is an Ore-collapsible subset of $G$. By Lemma~\ref{OreCollToGem}, there exists a diamond or emerald of $G$ whose vertices lie in $R'$, contradicting that $G$ is ungemmed. 
\end{proof}

\begin{corollary}\label{NoColl}
If $G$ is a tight ungemmed graph, then there does not exist a collapsible susbet of $G$.
\end{corollary}

\begin{lemma}\label{ClusterInAK4}
If $G$ is a tight ungemmed graph, and $u,v$ are degree four and there exists a subgraph $H$ of $G$ isomorphic to $K_4$ such that $u,v \in V(H)$, then $u$ and $v$ are in the same cluster.
\end{lemma}
\begin{proof}
Suppose for a contradiction that there exists $H=K_4$ containing two vertices $u,v$ of degree four not in the same cluster. Let $a$ be the neighbor of $u$ not in $H$. Let $b$ be the neighbor of $v$ not in $H$. Since $u$ and $v$ are not in the same cluster, $a\ne b$. Yet $G\setminus V(H) + ab$ is not $4$-colorable as otherwise $G$ is $4$-colorable. Thus $a,b$ is an identifiable pair in $V(G)\setminus V(T)$, contradicting Lemma~\ref{NoIdPair}.
\end{proof}

We can now strengthen the outcome of Lemma~\ref{PotCore1} as follows:

\begin{lemma}\label{PotCore4}
If $G$ is a tight ungemmed graph, $R\subsetneq V(G)$, $|R|\ge 5$ , then $p_G(R) > p(G)+7+\delta+3\epsilon+Q$ unless $G\setminus R$ is a single vertex of degree four in $G$.
\end{lemma}
\begin{proof}
Suppose there exists $R\subsetneq V(G)$, $|R|\ge 5$ with $p_G(R)\le p(G)+7+\delta+3\epsilon+Q$. By Corollary~\ref{NoColl}, $R$ is not collapsible. It follows then from Lemma~\ref{PotCore1} that $p_G(R)> p(G)+7-4\delta+4\epsilon$. As $R$ is not collapsible, it follows from Lemma~\ref{Coll} that there exists an extension $R'$ of $R$ with extender $W$ and core $X$ where either the extension is not spanning, not complete, or has a core of size at least two. Yet following the calculations in the proof of Lemma~\ref{PotCore1}, we find that the extension is spanning, complete, and has a core of size 4. Thus, $p(G)\le p_G(R)+p(W) - 12 +4\delta-4\epsilon$. That is, $p(W)\ge 12-4\delta +4\epsilon + p(G)-p_G(R)$. As $p(G)-p_G(R)\ge -7-\delta-Q-3\epsilon$, we find that that $p(W)\ge 5-5\delta+\epsilon-Q$. As $P> 5\delta-\epsilon+Q=6\delta-\epsilon$, it follows that $W$ is $5$-Ore. 

If $W\ne K_5$, then by Lemma~\ref{DisjointGem}, there is either a diamond $D$ or an emerald $E$ whose vertices lie in $W\setminus X$. As the extension is complete and spanning, the degree of a vertex in $W$ is the same as its degree in $G$. Hence $E$ is an emerald or $D$ is a diamond in $G\setminus R$, contradicting that $G$ is ungemmed. So we may assume that $W=K_5$ and thus $|W\setminus X|=1$. So $|G\setminus R|=1$ and as the extension is complete, $G\setminus R$ consists of a single vertex of degree four.
\end{proof}

\begin{lemma}\label{NoCluster2}
If $G$ is a tight ungemmed graph, then $G$ contains no cluster of size at least 2.
\end{lemma}
\begin{proof}
Suppose there exists a cluster $C$ of size at least two. If $|C|\ge 3$, then $C$ and its neighbors form a diamond, contradicting that $G$ is ungemmed. So we may suppose that $|C|=2$. Let $C=\{x,y\}$ and let $z_1,z_2,z_3$ be the other neighbors of $x$ (and hence of $y$). 

\begin{claim}\label{Zs}
For all $i,j,k$ where $\{i,j,k\} = \{1,2,3\}$, either 
\begin{enumerate}
\item $z_i$ is adjacent to $z_j$, or
\item $z_k$ has degree four, or
\item $z_i$ identified with $z_j$ has a $5$-Ore subgraph disjoint from $\{x,y,z_k\}$. 
\end{enumerate}
\end{claim}
\begin{proof}
To see this, suppose 1 does not hold, that is, $z_i$ and $z_j$ are not adjacent. Let $G'$ be obtained from identifying $z_i$ and $z_j$ to a new vertex $w$ and deleting $x$ and $y$. Note that $G'$ is not $4$-colorable as otherwise $G$ is $4$-colorable. Let $K$ be a $5$-critical subgraph of $G'$ and note that $w\in V(K)$. Let $R=(V(K)-w) \cup \{z_i,z_j,x,y\}$. Note that $|R|=|V(K)|+3$, $|E(G[R])|\ge |E(K)|+5$ and $T(R)\ge T(K\setminus\{w\})+1\ge T(K)-1$. Thus $p(R)\le p(K)+7+\delta+3\epsilon$.

As $G$ is good and $|V(K)|\le |V(G)|$, $p(K) \le 5 + 5\epsilon - 2\delta$. First suppose $z_k\in V(K)$. In that case, $|E(G[R])|\ge |E(K)|+7$ and hence $p(R)\le p(K)-1+\delta+3\epsilon$. Since $|V(K)|\ge 5$ as $K$ is $5$-critical, $|R|\ge 5$ and hence by Lemma~\ref{ExtDec}, $p(R) \ge p(G)$. If $R\ne V(G)$, then by Lemma~\ref{ExtDec}, $p(R) \ge p(G) + 4 - \delta + 4\epsilon$. But then $p(G)\le p(K) - 5 + 2\delta -\epsilon \le 4\epsilon$, contradicting that $G$ is tight since $5-P-Q > 4\epsilon$. So we may assume that $V(G)=R$. But then $p(G) = p(R) = p_{KY}(R) + \epsilon|R| - \delta T(R)$. As $G$ is not $5$-Ore, $p_{KY}(G) \le 2$ by Theorem~\ref{Ore5EqualityP}. Hence $p(G) \le 2 + \epsilon (|V(K)|+3) - \delta ( T(K) - 1) = p(K) - 3 + 3\epsilon +\delta$ since $p_{KY}(K)=5$ and $p(K) = p_{KY}(K) + \epsilon |V(K)| - \delta T(K)$. But now $p(G) \le 2 +8\epsilon -\delta$, a contradiction since $5-P-Q > 2+8\epsilon - \delta$ that is $P+Q < 3-8\epsilon+\delta$.

So we may assume that $z_k\not\in V(K)$. If $K$ is $5$-Ore, then 3 holds as desired. So we may suppose that $K$ is not $5$-Ore. As $G$ is good, it follows that $p(K)\le 5-P$.
If $z_k$ is a vertex of degree 4, then 2 holds as desired. So we may assume that $z_k$ does not have degree 4 and hence by Lemma~\ref{PotCore4} that $p(R) > p(G) + 7 + \delta + 3\epsilon + Q$. As $G$ is tight, $p(G)\ge 5-P-Q$. So $p(R) > 12 + \delta +3\epsilon - P$. Yet $p(R) \le p(K) + 7 + \delta + 3\epsilon \le 12-P+\delta+3\epsilon$, a contradiction.
\end{proof}

If at least two of the pairs $i\ne j\in\{1,2,3\}$ satisfy 1 in Claim~\ref{Zs}, then $G$ contains a subgraph isomorphic to $K_5-e$, contradicting Lemma~\ref{NoIdPair}. If at least two of pairs $i\ne j\in\{1,2,3\}$ satisfy 2 in Claim~\ref{Zs}, then either $G$ contains a $K_4-e$ subgraph $H$ of degree fours which is impossible in a $5$-critical graph (every coloring of $G\setminus H$ extends to $G$), or, $G$ contains an emerald contradicting that $G$ is ungemmed.

Thus if only one of the pairs $i\ne j\in\{1,2,3\}$ satisfies 3 in Claim~\ref{Zs}, then we may assume without loss of generality $i=1,j=2$ satisfies 3, $i=1,j=3$ satisfies 2 and $i=2,j=3$ satisfies 1. That is, $z_1$ identified with $z_2$ has a $5$-Ore subgraph disjoint from $\{x,y,z_3\}$, $z_2$ is degree four and $z_2$ is adjacent to $z_3$. But then $\{x,y,z_2,z_3\}$ induces a subgraph isomorphic to $K_4$. By Lemma~\ref{ClusterInAK4}, $z_2$ is in the same cluster as $x$ and $y$. But then, there exists a cluster of size at least 3, a contradiction as above.

So at least two of the pairs say $i=1,j=3$ and $i=2,j=3$ satisfy 3 in Claim~\ref{Zs}. Let $K_1$ be the $5$-Ore graph obtained when identifying $z_1$ and $z_3$ to a new vertex $w_1$ and $K_2$ be the $5$-Ore graph obtained when identifying $z_2$ and $z_3$ to a new vertex $w_2$. Recall that $z_2\not\in V(K_1)$ and $z_1\not\in V(K_2)$. Moreover, for each $i\in\{1,2\}$, $p_{KY}(K_i)=5$.

Let $R_1=(V(K_1)-w_1)\cup\{z_1,z_3\}$ and $R_2=(V(K_2)-w_2)\cup\{z_2,z_3\}$. Let $H=R_1\cap R_2$ and let $R=R_1\cup R_2\cup\{x,y\}$. Note that $|R|\ge 9$. It follows from Lemma~\ref{Coll5Ore} that $p_{KY}(H)\ge 9$. Yet for $i\in\{1,2\}$, $p_{KY}(R_i)=14$. Thus $p_{KY}(R_1\cup R_2) \le 14 + 14 - 9 = 19$. But then $p_{KY}(R) = 19 + 9(2)-4(7) = 9$. 

Note that $T(R)\ge \max\{T(K_1),T(K_2)\} - 1$. By Lemma~\ref{5OrePot}, it follows that for each $i\in\{1,2\}$, $T(K_i)\ge 2+\frac{|V(K_i)|-1}{4}$. Yet $|V(K_1)|+|V(K_2)|\ge |R|-3$. So there exists $i\in\{1,2\}$ such that $|V(K_i)|\ge \frac{|R|-3}{2}$. Hence $T(R)\ge \max\{T(K_1)-1,T(K_2)-1\} \ge 1 + \frac{|R|-5}{8}$. Thus $p(R) \le 9 +\epsilon |R| - \delta (1+\frac{|R|-5}{8})$. If $R\ne V(G)$, then $p(R)\le 9$ since $\delta\ge 8\epsilon$. By Lemma~\ref{PotCore1}, $R$ is a collapsible subset of $G$, contradicting Corollary~\ref{NoColl}. 

So we may assume that $R=V(G)$. But then by Theorem~\ref{Ore5EqualityP} as $G$ is not $5$-Ore, it follows that $p_{KY}(G)\le 2$. Thus $p(G)\le 2 + \epsilon |V(G)| - \delta (1+\frac{|V(G)|-5}{8})\le 2-3\epsilon$ since $\delta \ge 8\epsilon$, contradicting that $G$ is tight since $P+Q =7\delta < 3+3\epsilon$. 

\end{proof}

\section{Properties of a Minimum Counterexample}

For the rest of this section, let $G$ be a good graph that is not $5$-Ore with $p(G)>5-P$.

\begin{lemma}\label{NoDiamond}
$G$ is $3$-connected and hence contains no diamond.
\end{lemma}
\begin{proof}
Suppose $G$ is not $3$-connected. Hence there exists a $2$-cut $x,y$ of $G$. That is $G$ is the Ore-composition of two graphs $G_1$ and $G_2$. As $G$ is not $5$-Ore, at least one of $G_1$,$G_2$ is not $5$-Ore. By Lemma~\ref{5OreT}, $T(G)\ge T(G_1)+T(G_2)-2$. Recall that $|V(G)|=|V(G_1)|+|V(G_2)|-1$ and $|E(G)|=|E(G_1)|+|E(G_2)|-1$. Thus $p(G) \le p(G_1)+p(G_2)-5-\epsilon + 2\delta$. Suppose without loss of generality that $G_2$ is not $5$-Ore. Since $G$ is good, $p(G_2)\le 5-P$. So $5-P < p(G)\le p(G_1)- P+2\delta-\epsilon$. So $p(G_1)\ge 5-2\delta + \epsilon$. Thus $G_1$ is $5$-Ore and indeed $G_1=K_5$ by Lemma~\ref{5OrePot}. Since $G_1=K_5$, $T(G)\ge T(G_1)+T(G_2)-1$ by Lemma~\ref{5OreT}. So $p(G)\le p(G_1)+p(G_2)-5-\epsilon+\delta = p(G_2)+4\epsilon-\delta \le 5-P$ since $\delta\ge 4\epsilon$, a contradiction.
\end{proof}

\begin{lemma}\label{NoIdPairMin}
There exists no identifiable pair in a proper subset of $V(G)$.
\end{lemma}
\begin{proof}
Note that Lemma~\ref{CollSmall} holds for $G$ as $G$ is tight. We now repeat the proof of Lemma~\ref{NoIdPair} except for when $W$ is $5$-Ore. Note that the case when $K$ is $5$-Ore yields a contradiction because $G$ has no Ore-collapsible set as $G$ is $3$-connected by Lemma~\ref{NoDiamond}. Thus we may assume that $K=R+uv$ is not $5$-Ore. Recall that $R$ is collapsible and $W$ has a core of size one. By Lemma~\ref{Extension}, $p(G) \le p_G(R) + p(W) - 9 - \epsilon + \delta$. Yet $p_G(R) \le p(K)+4+\delta$. As $G$ is good and $K$ is not $5$-Ore, $p(K)\le 5-P$. Hence $p(G) \le p(W) - P - \epsilon + 2\delta$. Yet $p(G) > 5-P$. So $p(W) \ge 5 - 2\delta + \epsilon$. By Lemma~\ref{5OrePot}, $W=K_5$. But in that case $T(W)=T(W-X)$ and hence $p(G) \le p_G(R) + p(W) - 9 - \epsilon \le p(K) + 4 +\delta + 5 + 5\epsilon - 2\delta - 9 -\epsilon = p(K) +4\epsilon - \delta \le p(K) \le 5-P$ since $\delta\ge 4\epsilon$, a contradiction.
\end{proof}

\begin{corollary}\label{NoK5e}
$G$ does not contain $K_5-e$ as a subgraph.
\end{corollary}

\begin{lemma}\label{Ungemmed}
$G$ contains no emerald and hence is ungemmed.
\end{lemma}
\begin{proof}
Suppose not and let $E$ be an emerald of $G$. As $G\ne K_5$, there exist vertices $a,b\in E$ such that $a$ and $b$ are not in the same cluster. Let $u$ be the neighbor of $a$ not in $E$ and let $v$ be the neighbor of $b$ not in $E$. Hence $u\ne v$. But now $u$ and $v$ are an identifiable pair in $V(G)-V(E)$ contradicting Lemma~\ref{NoIdPairMin}.
\end{proof}

\subsection{Almost $5$-Ore Graphs}

Here is a crucial definition.

\begin{definition}
A graph is \emph{almost $5$-Ore} if it can be obtained from a $5$-Ore graph by deleting a vertex in a cluster of size at least two. We call any other vertex in that cluster \emph{special}.
\end{definition}

\begin{definition}
We define $D_4(G)$ to be the subgraph of $G$ induced by the vertices of degree four.
\end{definition}

The next lemma is useful in finding almost $5$-Ore subgraphs in $G$.

\begin{lemma}\label{Special}
If $uv$ is an edge of $D_4(G)$, then $u$ is a special vertex of an almost $5$-Ore subgraph of $G-v$. 
\end{lemma}
\begin{proof}
Let $uv$ be an edge of $D_4(G)$. Let $G'$ be obtained from $G$ by deleting $v$ and adding a new vertex $u'$ adjacent to $u$ and the neighbors of $u$. Note that $u,u'$ are degree four in $G'$. Moreover, $G'$ is not $4$-colorable as otherwise a $4$-coloring of $G'$ can be extended to a $4$-coloring of $G$ by coloring $v$ and then coloring $u$ with a color of $u$ or $u'$ different from the color of $v$. Let $K$ be a $5$-critical subgraph of $G'$ and let $R=K-u'$. Note that $p(R)\le p(K) + 7 + \delta$. By Lemma~\ref{PotCore4}, it follows that either $K$ is $5$-Ore or that $G\setminus R$ is a single vertex of degree four, namely, $v$. 
If $K$ is $5$-Ore, then $u$ is a special vertex of $K-u'$ which is an almost $5$-Ore subgraph of $G-v$ as desired.

So we may suppose that $K$ is not $5$-Ore and hence that $G\setminus R$ is a single vertex of degree four. It follows that $p(K)\ge p(G)-\delta$ and hence $K$ is tight as $Q\ge \delta$. If $K$ is ungemmed, then $\{u,u'\}$ is a cluster of size two in $K$, contradicting Lemma~\ref{NoCluster2}. 
 
So we may assume that $K$ has either a diamond $D$ or an emerald $E$. First suppose that $K$ contains a diamond $D$. Then $u'\in V(D)$ as otherwise $G$ contains a subgraph isomorphic to $K_5-e$, contradicting Lemma~\ref{NoK5e}. We may as well assume then without loss of generality that $u\in V(D)$ and that $u,u'$ are degree four in $D$. Let $x$ be the other vertex of degree four in $D$. Now $x$ must have degree at least four in $G$. Hence $x$ is adjacent to not only the vertices of $D-u'$ but also to $v$. But then $u$ and $x$ are in a cluster in $G$, contradicting Lemma~\ref{NoCluster2} since $G$ is tight and $G$ is ungemmed by Lemma~\ref{Ungemmed}.

So we may suppose that $K$ contains an emerald $E$. If $u'\in E$, then we may assume that $u$ is also in $E$. But then the other vertices of $E-\{u,u'\}$ must be adjacent to $v$ since they have degree at least four in $G$. So $E-u'+v$ is in fact an emerald in $G$, contradicting Lemma~\ref{Ungemmed}. So we may assume $u'\not\in E$. It follows that $E$ is a subgraph of $G$ isomorphic to $K_4$. Let $X$ be the set of vertices of $E$ of degree four in $G$. By Lemma~\ref{ClusterInAK4}, $E$ is a cluster. By Lemma~\ref{NoCluster2}, $|E|\le 1$. Thus $E-X$ are vertices of degree at least five in $G$ but degree four in $G'$. Hence every vertex in $E-X$ must be adjacent to $v$. But then $E\cup v$ has a subgraph isomorphic to $K_5-e$, contradicting Lemma~\ref{NoK5e}.

\end{proof}

Next we need the following general propositions about $5$-Ore graphs but first a definition.

\begin{definition}
Let $H$ be an almost $5$-Ore graph and let $w$ be a special vertex of $H$. A \emph{frame} of $H$ with special vertex $w$ is a graph $J$ isomorphic to $K_4$ such that $V(J)=w\cup N_H(w)$ and $H$ can be obtained from $J$ by Ore-composition with $5$-Ore graphs whose replaced edges are in $E(J)$. We call the vertices of $J$ the \emph{corners} of the frame and the graphs used for the Ore-composition the \emph{bars}.
\end{definition}

\begin{proposition}\label{K4Frame}
If $H$ is $5$-Ore and $wv\in E(H)$ such that $w$ has degree four in $H$ and there does not exist an identifiable pair in $H-v$, then there exists an almost $5$-Ore subgraph $H'$ of $H-v$ such that $w$ is the special vertex and a frame $J$ of $H'$ with special vertex $w$.
\end{proposition}
\begin{proof}
We proceed by induction on $|V(H)|$. If $H=K_5$, the lemma follows with $H'=J=H-v$. So we may suppose $H\ne K_5$. As $H$ is $5$-Ore, then $H$ is the Ore-composition of two $5$-Ore graphs $H_1$ and $H_2$. Without loss of generality suppose that $H_1$ is the edge-side with replaced edge $xy$ and $H_2$ is the vertex-side of the composition with split vertex $z$. 

Since there does not exist a collapsible subset of $H-v$, it follows that $v\in V(H_1)$. First suppose $w\in V(H_2)$. By induction on $H_2$, there exists an almost $5$-Ore subgraph $H'$ of $H_2-z$ with frame $J$ and hence of $H-v$ as desired. Next suppose $w\in V(H_1)\setminus \{x,y\}$. By induction on $H_1$, there exists an almost subgraph $H'$ with frame $J$ of $H_1-v$. If $xy\not\in E(H')$, then $H'$ is a subgraph of $H-v$ as desired. So we may assume that $xy\in E(H')$. Then let $H''$ be obtained from $H'$ by an Ore-composition with $H_2$ whose replaced edge is $xy$. Hence $H''$ is a subgraph of $H-v$ as desired.

So we may assume that $w=x$ without loss of generality. Hence $v\ne y$ since $wv \in E(H)$. Since $w$ has degree at least four in $H_1$ and at least one neighbor in $H_2$, we find that $w$ has degree exactly four in $H_1$ and exactly one neighbor, call it $w'$, in $H_2$. But then $w',y$ is an identifiable pair in $H-v$, a contradiction. 
\end{proof}

\begin{proposition}\label{ReplacedK5}
If $H$ is an almost $5$-Ore graph with special vertex $w$ and there does not exist an identifiable pair in $H$, then there exists a frame of $H$ with special vertex $w$.
\end{proposition}
\begin{proof}
Let $H$ be obtained from a $5$-Ore graph $K$ by deleting a vertex $v$ in a cluster of size at least two. Note that $wv\in E(K)$ and $w$ has degree four in $H$. Moreover, there does not exist an identifiable pair in $H-v$. By Proposition~\ref{K4Frame}, there exists an almost $5$-Ore subgraph $K'$ of $K-v$ with frame $J$ whose special vertex is $w$. But then $K'\cup \{w\}$ is a $5$-Ore subgraph of $K$ and hence $K=K'\cup \{w\}$. Thus $H=K'$ and $J$ is a frame of $H$ as desired. 
\end{proof}

\begin{lemma}\label{InTheBar}
Let $H$ be an almost $5$-Ore graph with frame $J$ such that there does exist an identifiable pair in any proper subset of $H$. If $R\subsetneq V(H)$ such that $p_{KY}(R)=12$, then there exists a vertex $v\in V(J)$ and an edge $e=uv$ of $J$ such that $R$ is a subset of $V(H')\cup\{v\}$ where $H'$ is the bar of $e$.
\end{lemma}
\begin{proof}
Suppose not. Let $X=V(J)\cap R$ and for every $e=ww'\in E(J)$, let $R_e = R\cap V(H_e)\cup \{w,w'\}$ where $H_e$ is the bar of $e$. Since there does not exist an identifiable pair in $R_e$, then by Lemma~\ref{Coll5Ore}, $p_{XY}(R_e)\ge 12$ if $|R_e|\ge 2$. Moreover, if both $w,w'$ in $R_e$, then $p_{KY}(R_e) \ge 14$ since $R_e\setminus \{w,w'\} \cup z$ is a subset of $H_e$ where $z$ is the identification of $w,w'$ in the bar $H_e$ and $H_e$ is $5$-Ore; if in addition $w,w'$ in $R_e$ and $R_e\ne H_e-z \cup \{w,w'\}$, then $p_{KY}(R_e) \ge 18$ by Lemma~\ref{Coll5Ore}.

Given these calculations, it now follows by summing $p(R_e)$ for all $e\in E(J)$ and subtracting the potential of overcounted vertices of $J$ that $p_{KY}(R) \ge 9|X| - 4 |E(J[x])| = p_{KY}(J[X])$. Since $p_{KY}(R) = 12$, we have that either $|X|=4, 1$ or $0$. First suppose $|X|=4$. Since $R$ is a proper subset of $H$, then for at least one edge $e$ in $J$, $R_e\ne H_e-z \cup \{w,w'\}$ and hence there is an additional $+4$ in the count above so that $p_{KY}(R) \ge 12 + 4 = 16$, a contradiction. Next suppose $|X|=1$ and let $X=\{v\}$. We may assume that $R-v$ intersects at least two bars of $J$ as otherwise we have a contradiction. But then $p_{XY}(R) \ge 2\cdot 12- 9 = 15$, a contradiction. So we may assume that $|X|=0$. Once again we may assume that $R$ intersects at least two bars of $J$, but then $p_{KY}(R) \ge 2\cdot 9 = 18$, a contradiction.
\end{proof}

\begin{lemma}\label{Size2}
Every component of $D_4(G)$ has size at most 2.
\end{lemma}
\begin{proof}
We now prove that every component of $D_4(G)$ has size at most 2. Let $uvw$ be a path on three vertices in $D_4(G)$. By Lemma~\ref{Special}, $v$ is a special vertex of an almost $5$-Ore subgraph $H_1$ of $G-w$ and $v$ is a special vertex of an almost $5$-Ore subgraph $H_2$ of $G-u$. Since there does not exist an identifiable pair in $H_1$ by Lemma~\ref{NoIdPairMin}, there exists a frame $F_1$ of $H_1$ by Proposition~\ref{ReplacedK5}. By symmetry, there exists a frame $F_2$ of $H_2$. Note that $p_{KY}(H_1)=p_{KY}(H_2)=12$. 

Note that $u\not\in V(H_2)$ and $w\not\in V(H_1)$. However, the other two neighbors of $v$, call them $v_1$ and $v_2$, are in both $H_1$ and $H_2$ since $d_{H_1}(v), d_{H_2}(v)\ge 4$. Moveover, $v$ is in both $F_1$ and $F_2$. Let $V(F_1)=\{u,v,v_1',v_2'\}$. Since $G$ does not have an identifiable pair by Lemma~\ref{NoIdPairMin}, it follows that either $v$ has two neighbors in the bar of $vv_1'$ or that $v_1'$ is a neighbor of $v$. Since $v$ has degree four, it follows that $v_1'$ is a neighbor of $v$ and similarly $v_2'$ is a neighbor of $v$. Hence $v_1,v_2$ are each in both of $V(F_1)$ and $V(F_2)$.

Let $R=V(H_1)\cap V(H_2)$. Since there does not exist an identifiable pair in $G$ by Lemma~\ref{NoIdPairMin}, $R$ is not collapsible in $H_1$ and hence by Lemma~\ref{Coll5Ore} applied to the $5$-Ore graph obtained from $H_1$ by cloning $v$, $p_{KY}(R)\ge 12$. Since $w\not\in V(H_1)$, $R$ is a proper subset of $V(H_1)$. But now applying Lemma~\ref{InTheBar} to $H_1, F_1$ and $R$ we find that $p_{KY}(R) \ge 13$ since $R\cap V(F_1) = \{v,v_1,v_2\}$. 


Let $R'=V(H_1)\cup V(H_2)$. Note that $|R|\ge 5$. Now $p_{KY}(R')\le p_{KY}(H_1)+p_{KY}(H_2)-p_{KY}(R) \le 12+12-13 = 11$. By Lemma~\ref{5OrePot}, it follows that for $i\in\{1,2\}$, $T(H_i')\ge 2+\frac{|V(H_i')|-1}{4}$. Thus for each $i\in\{1,2\}$, $T(H_i)\ge 1+\frac{|V(H_i)|}{4}$. Yet $|V(H_1)|+|V(H_2)|\ge |R'|+1$. Thus there exists $i$ such that $|V(H_i)|\ge \frac{|R'|+1}{2}$. Hence $T(R')\ge \max\{T(H_1),T(H_2)\} \ge 1 + \frac{|R'|+1}{8}$. 

So $p(R')\le p_{KY}(R') + \epsilon|R'| - \delta(1+\frac{|R'|+1}{8}) \le p_{KY}(R') -9\epsilon \le 11 - 9\epsilon$ since $\delta \ge 8 \epsilon$. If $R'\ne V(G)$, then by Lemma~\ref{PotCore4}, $p(R') > p(G) + 7 + \delta + 3\epsilon + Q$ unless $G\setminus R'$ is a single vertex of degree four. The former case implies that $p(G) \le 4-\delta - 12\epsilon - Q \le 5-P$ since $P < 1+\delta + Q +12\epsilon$ as $\epsilon < 1/20$, contradicting that $p(G) > 5-P$. So suppose the latter case. Then $p(G) \le p(R') + (9+\epsilon) - 4(4) \le p_{KY}(R') - 7 - 8 \epsilon$. Yet $p_{KY}(R') = p_{KY}(G) + 7$ since $G\setminus R'$ is a single vertex of degree four. Hence $p(G) \le p_{KY}(G) - 8\epsilon$. Since $G$ is not $5$-Ore, then $p_{KY}(G)\le 2$ by Theorem~\ref{Ore5EqualityP}. Hence $p(G) \le 2-8\epsilon$ contradicting that $p(G) > 5-P$ since $P < 3+8\epsilon$.

So we may assume that $R'=V(G)$. But then $p_{KY}(R')=p_{KY}(G)\le 2$ by Theorem~\ref{Ore5EqualityP} since $G$ is not $5$-Ore. Hence $p(R)=p(G)\le p_{KY}(R')-9\epsilon\le 2 -9\epsilon$, contradicting that $p(G) > 5-P$ since $P < 3+9\epsilon$.
\end{proof}

\begin{lemma}\label{Deg5OnlyOneMatched4}
If $v$ is a vertex of degree $5$ in $G$, then $v$ has at most one neighbor of degree $4$ that is incident with an edge of $D_4(G)$.
\end{lemma}
\begin{proof}
Suppose not and let $u,w$ be neighbors of $v$ of degree four incident with an edge of $D_4(G)$. Let $u',w'$ be the other ends of those edges respectively (note that $u$ and $w$ may be adjacent in which case $u'=w$ and $w'=u$). By Lemma~\ref{Size2}, $u$ is in an almost $5$-Ore graph $H_1$ not containing $u'$ and $w$ is in an almost $5$-Ore graph $H_2$ not containing $w'$. As $u$ has degree three in $H_1$, $v\in V(H_1)$ and similarly $v\in V(H_2)$. Also, $p_{KY}(H_1)=p_{KY}(H_2)=12$. Let $R=V(H_1)\cap V(H_2)$. 

\begin{claim}\label{13}
$p_{KY}(R) \ge 13$.
\end{claim}
\begin{proof}
Suppose not.  So we may suppose that $p_{KY}(R)\le 12$. Note that $v\in V(H_1)\cap V(H_2)$. Note that $d_{H_1}(v), d_{H_2}(v) \ge 3$ and hence $N(v) \cap R \ne \emptyset$. 

Suppose $|N(v)\cap R| \le 2$. So $p_{KY}(R - v) \le p_{KY}(R) -9+4(2) = p_{KY}(R) - 1 = 12 - 1 = 11$. But then by Lemma~\ref{Coll5Ore}, $R - v$ is collapsible in the $5$-Ore graph obtained from $H_1$ by cloning $u$ and so $R-v$ has an identifiable pair contradicting Lemma~\ref{NoIdPairMin}. Thus we may assume that $|N(v)\cap R|\ge 3$. Similarly by Lemma~\ref{Coll5Ore}, it follows that $p_{KY}(R) = 12$ by Lemma~\ref{Coll5Ore}.

By Lemma~\ref{ReplacedK5}, there exists a frame $F_1$ of $H_1$ with special vertex $u$ and similarly there exists a frame $F_2$ of $H_2$ with special vertex $w$. 

We consider two cases. First suppose that both $u\not\in V(H_2)$ and $w\not\in V(H_1)$. Hence $u,w\not\in R$. Note that $v\in V(F_1)\cap V(F_2)$. Since $G$ does not have an identifiable pair by Lemma~\ref{NoIdPairMin}, then by Lemma~\ref{InTheBar} applied to $R$ and $F_1$ we find that there exists $e\in E(F_1)$ such that $R-v$ is contained in the bar of $e$ in $F_1$. Let $V(F_1) = \{u,v,v_1,v_2\}$. We may assume without loss of generality that $e=vv_1$. However, then $d(v)\ge 6$ since $v$ is adjacent to both $u$ and $w$, has three neighbors in $R$, and is either adjacent to $v_2$ or has a neighbor in the bar of $vv_2$. This is a contradiction since $v$ has degree five in $G$.

So we may assume without loss of generality that $w\in V(H_1)$ and hence $w\ne u'$ so that $w$ is not adjacent to $u$. Since $V(F_1)=N_G(u)\cup\{u\}-u'$, we find that $w\not\in V(F_1)$. But then $d_{H_1}(w) = 4$ and hence $N_G(w) \subseteq V(H_1)$. In particular, $w'\in V(H_1)$ and hence $R$ is a proper subset of $V(H_1)$ as $w'\not\in R$. Since $G$ does not have an identifiable pair by Lemma~\ref{NoIdPairMin}, then by Lemma~\ref{InTheBar} applied to $R$ and $F_1$ we find that there exists $e\in E(F_1)$ such that $R-v$ is contained in the bar of $e$ in $F_1$. Let $V(F_1) = \{u,v,v_1,v_2\}$. We may assume without loss of generality that $e=vv_1$. Note then that $u\not \in R$ and hence $u\not\in V(H_2)$.

Let $H'$ be the bar of $e$ and let $z$ denote the vertex that is the identification of $v$ and $v_1$. Recall that $H'$ is $5$-Ore and note that $wz\in E(H')$. Further note that $w$ has degree four in $H'$ and that there does not exist an identifiable pair in $H'-z$ since there does not exist an identifiable pair in $G$ by Lemma~\ref{NoIdPairMin}. Now, by Proposition~\ref{K4Frame} applied to $H'$, $w$ and $z$, we find that there exists an almost $5$-Ore subgraph $H_3$ of $H'-z$ with special vertex $w$ and frame $F_3$.

Now note that $V(F_2) = N_G(w)\cup\{w\}-w'$ and $V(F_3) = N_G(w)\cup\{w\} - v$. Now consider $S=V(H_2)\cup V(H_3)$. Note that $S\ne V(G)$ since $u\not \in S$ as $u\not\in V(H_3)$ since $H_3\subseteq H'-z$ and $u\not\in V(H_2)$ as noted above. We claim that $w',v$ is an identifiable pair in $S$, contradicting Lemma~\ref{NoIdPairMin}. Suppose not and let $\phi$ be a $4$-coloring of $S+vw'$. To see this, note that in every $4$-coloring of an almost $5$-Ore subgraph with a frame, the corners of the frame must receive different colors. Now $\phi$ induces a $4$-coloring of $H_2$ and hence $V(F_2)$ all receive different colors in $\phi$. Similarly, $\phi$ induces a $4$-coloring of $H_3$ and hence $V(F_3)$ all receive different colors in $\phi$. Since $\phi(v)\ne \phi(w')$, we find that all of $V(F_2)\cup V(F_3) = N_G(w) \cup \{w\}$ receive different colors in $\phi$, a contradiction since $\phi$ is a $4$-coloring and $w$ $|V(F_2)\cup V(F_3)|=5$. This proves the claim that $w',v$ is an identifiable pair and so concludes our proof.
\end{proof}

Let $R'=V(H_1)\cup V(H_2)$. Since $p_{KY}(R)\ge 13$ by Claim~\ref{13}, then $p_{KY}(R') \le p_{KY}(H_1) + p_{KY}(H_2) - p_{KY}(R) \le 12 + 12 - 13 = 11$.

By Lemma~\ref{5OrePot}, it follows that for $i\in\{1,2\}$, $T(H_i)\ge 1+\frac{|V(H_i)|}{4}$. Yet $|V(H_1)|+|V(H_2)|\ge |R'|+1$. Thus there exists $i$ such that $|V(H_i)| \ge \frac{|R'|+1}{2}$. Hence $T(R')\ge \max\{T(H_1),T(H_2)\} \ge 1+\frac{|R'|+1}{8}$. So $p(R')\le p_{KY}(R') + \epsilon|R| - \delta(1+ \frac{|R|+1}{8}) \le p_{KY}(R') -9\epsilon \le 11 - 9\epsilon$ since $\delta\ge 8\epsilon$. 

If $R'\ne V(G)$, then by Lemma~\ref{PotCore4}, $p(R') > p(G) + 7 + \delta + 3\epsilon + Q$ unless $G\setminus R'$ is a single vertex of degree four. The former case implies that $p(G) \le 4-\delta - 12\epsilon - Q$, contradicting that $p(G) > 5-P$ since $P\le 1+\delta+Q+12\epsilon$ and $\epsilon < 1/20$. So suppose the latter case. Then $p(G) \le p(R') + (9+\epsilon) - 4(4) \le p_{KY}(R') - 7 - 8 \epsilon$. Yet $p_{KY}(R') = p_{KY}(G) + 7$ since $G\setminus R'$ is a single vertex of degree four. Hence $p(G) \le p_{KY}(G) - 8\epsilon$. Since $G$ is not $5$-Ore, then $p_{KY}(G)\le 2$ by Theorem~\ref{Ore5EqualityP}. Hence $p(G) \le 2-8\epsilon$ contradicting that $p(G) > 5-P$ since $P < 3+8\epsilon$.

So we may assume that $R'=V(G)$. But then $p_{KY}(R')=p_{KY}(G)\le 2$ by Theorem~\ref{Ore5EqualityP} since $G$ is not $5$-Ore. Hence $p(R')=p(G)\le p_{KY}(R') - 9\epsilon \le 2 - 9\epsilon$ from above, contradicting that $p(G) > 5-P$ since $P < 3+9\epsilon$.

\end{proof}

\section{Discharging}

In this section, we prove Theorem~\ref{Main}. We will need the following theorem of Kierstead and Rabern~\cite{KiersteadRabern}, but first a definition.

\begin{definition}
The \emph{maximum independent cover number} of a graph $G$, denoted ${\rm mic}(G)$, is the maximum of $\sum_{v\in I}d(v)$ over all
independent sets $I$ of $G$.
\end{definition}

\begin{theorem}\label{mic}
If $G$ is a $k$-critical graph, then

$$|E(G)|\ge \frac{k-2}{2} |V(G)| + \frac{1}{2} {\rm mic}(G).$$
\end{theorem}

Let $G$ be a minimum counterexample to Theorem~\ref{Main} as in the previous section. We proceed by discharging. Let the charge of a vertex $v$, denoted $ch(v)$ be given by:

$$ch(v) = (9+\epsilon) - 2d(v).$$

We now discharge according to the following rule to obtain a new charge, denoted $ch_F(v)$.\\

{\bf Discharging Rule:} If $v$ is a vertex of degree at least 5 with a neighbor $u$ of degree four in a component of $D_4(G)$ of size at least two, then $v$ receives $+1/4$ charge from $u$.\\

\begin{lemma}
If $v$ has degree at least $5$, then $ch_F(v)\le -\frac{3}{4} + \epsilon$.
\end{lemma}
\begin{proof}
If $v$ has degree $5$, then $ch(v)=-1+\epsilon$. By Lemma~\ref{Deg5OnlyOneMatched4}, $v$ receives charge from at most one neighbor. Hence $ch_F(v) \ge -1 + \epsilon + 1/4 = \epsilon-\frac{3}{4}$ as desired.

Suppose then that $v$ has degree at least 6. Now, $ch(v) = (9+\epsilon)-2d(v)$ and $v$ receives at most $+1/4$ charge from each neighbor. Hence $ch_F(v) \le (9+\epsilon)-2d(v) + \frac{d(v)}{4} = 9+\epsilon -\frac{7}{4}d(v)$. As $d(v) \ge 6$, this is at most $-1+\epsilon$ as desired. 
\end{proof}

However, if $v$ has degree four and is in a component of size two of $D_4(G)$, then $ch_F(v)= \frac{1}{4} + \epsilon$. Meanwhile if $v$ is degree four and in a component of size 1 of $D_4(G)$, then $ch_F(v)=1+\epsilon$. Let $S$ be the number of components of size one in $D_4(G)$ and $M$ be the number of components of size two in $D_4(G)$. Hence the number of vertices of degree four is $S+2M$ and the number of vertices of degree at least five is $|V(G)|-S-2M$. Note that there is an independent set consisting of vertices of degree four of size at least $S+M$. Hence ${\rm mic}(G)\ge 4(S+M)$. Thus by Theorem~\ref{mic}

$$|E(G)|\ge \frac{3}{2}|V(G)|+ 2(S+M).$$

As $p(G) > 0$, $(9+\epsilon)|V(G)| > 4|E(G)|$. Hence $(9+\epsilon)|V(G)| > 6|V(G)| + 8(S+M)$. Thus 

$$S+M < \frac{3+\epsilon}{8}|V(G)|.$$

On the other hand, 

$$\sum_v ch(v) = (9+\epsilon)|V(G)| - 2\sum_v d(v) = (9+\epsilon)|V(G)|-4|E(G)| \ge p(G) > 0.$$

Hence $\sum_v ch_F(v) > 0$. Yet, $\sum_v ch_F(v) \le -\frac{3}{4}(|V(G)|-S-2M) + S + \frac{M}{4} + \epsilon |V(G)|$. Thus,

$$\frac{7}{4}(S + M) > \left(\frac{3}{4} - \epsilon\right)|V(G)|.$$

So on the one hand, $S+M > \frac{4}{7}(\frac{3}{4}-\epsilon)|V(G)|$ and on the other $S+M < \frac{3+\epsilon}{8}|V(G)|$. That is,

$$ \frac{3+\epsilon}{8}|V(G)| > \frac{3-4\epsilon}{7}|V(G)|.$$

Hence, $21+7\epsilon > 24-32\epsilon$. That is, $39\epsilon > 3$. So $\epsilon > \frac{1}{13}$, a contradiction.


\begin{thebibliography}{99}
\bibitem{Postle} L. Postle, The edge-density of $4$-critical graphs of girth 5 is $\frac{5}{3}+\epsilon$, manuscript.
\end{thebibliography}
\end{document}